\documentclass[a4]{article}
\usepackage{amsmath}
\usepackage{amsthm}
\usepackage{eufrak}

\begin{document}
\title{\bf Local probabilities for random permutations without long cycles}

\author { Eugenijus Manstavi\v cius and Robertas Petuchovas}
\maketitle
\footnotetext{{\it AMS} 2000 {\it subject classification.} Primary 05A15, secondary: 60C05, 11B75, 60F10. \break {\it Key words and
phrases}. Random permutation, cycle structure, involution, saddle point method, Dickman function}

\maketitle

\begin{abstract} We explore the probability $\nu(n,r)$ that
a permutation sampled from the symmetric group of order $n$
 uniformly at random has cycles of lengths not exceeding $r$, where  $1\leq r\leq n$ and $n\to\infty$.
Asymptotic formulas valid in specified   regions for the ratio $n/r$ are obtained using the saddle point method
combined with ideas originated in analytic number theory. Theorem 1 and its detailed proof are included to rectify formulas for small $r$ which have
been announced by a few other authors.

\end{abstract}

\newtheorem{theorem}{Theorem}
\newtheorem{lemma}{Lemma}
\newtheorem{corollary}{Corollary}
\newtheorem*{cor*}{Corollary}
\newtheorem{conj}{Conjecture}

\newtheorem{prop}{Proposition}
\newtheorem*{CLT}{CLT (\cite{EM-LMJ96})}
\newtheorem*{Comp}{Compactness Thm (\cite{EM-RJ08})}

\newtheorem*{WLLN}{WLLN (\cite{EM-RJ08})}

\def\E{\mathbf{E}}
\def\C{\mathbf{C}}
\def\D{\mathbf{D}}
\def\N{\mathbf{N}}
\def\R{\mathbf{R}}
\def\S{\mathbf{S}_n}
\def\Z{\mathbf{Z}}
\def\k{\kappa}
\def\e{\varepsilon}
\def\n{$n\to\infty$}
\def\cF{\mathcal F}

\def\re{{\rm e}}
\def\rO{{\rm O}}
\def\ro{{\rm o}}
\def\rd{{\rm d}}

\def\s{\smallskip}
\def\b{\bigskip}

\section{Introduction}

   The history on enumeration of  decomposable structures missing large components starts from the papers by K. Dickman and
   N.G. de Bruijn dealing with natural numbers composed of small prime factors. After numerous continuations, this analytic
   theory is now extensively developed and well exposed in the book by G. Tenenbaum \cite{GT} and in more recent papers.   By analogy, a similar theory was carried out for polynomials over a finite field (see, for example,  \cite{Odlyzko},  \cite{GarePana}) and   generalized to the so-called additive arithmetical semigroups (see \cite{Warlimont}, \cite{EM-Pal92}, \cite{EM-LMJ92}).  The survey \cite{Granv} discusses the parallelism between the theories. In no way, the list does not pretend to be complete, however,   it has influenced the present paper devoted to permutations. So far, the results on this particular class of structures do not reach the level of   research achieved for natural numbers. We focus only on permutations comprising the symmetric group $\S$ and seek asymptotic formulas for the probability $\nu(n,r)$
     that a permutation sampled uniformly at random has cycles of lengths not exceeding $r$, where  $1\leq r\leq n$, $r\in \N$, and $n\to\infty$.
   The goal is to cover the whole  range for the parameter $r$.

    Let us start from an exact formula.
      Denote $\N_0=\N\cup\{0\}$, $\ell_r(\bar s)=1s_1+\cdots+rs_r$, $\ell(\bar s)=\ell_n(\bar s)$, where $\bar s=(s_1,\dots, s_n)\in\N_0^n$.
   If $k_j(\sigma)$ equals the number of cycles in a permutation $\sigma\in\S$ of length $1\leq j\leq n$ and $\bar k(\sigma):=
   \big(k_1(\sigma),\dots,k_n(\sigma)\big)$ is the cyclic  structure vector, then (see, for example, \cite{ABT})
   \[
                \big|\{\sigma\in\S:\; \bar k(\sigma)=\bar s\}\big|={\mathbf 1}\{\ell(\bar s)=n\} n!\prod_{j=1}^n\frac{1}{j^{s_j} s_j!}.
   \]
   Hence
   \[
                \nu(n,r)=\frac{1}{n!}\big|\{\sigma\in\S:\; k_j(\sigma)=0\, \forall j\in\overline{r+1, n}\}\big|=
                \sum_{\ell_r(\bar s)=n}\prod_{j=1}^r\frac{1}{j^{s_j} s_j!},
   \]
   where the summation is over the vectors $\bar s\in\N_0^r$ with $\ell_r(\bar s)=n$.
   The formula can be rewritten in terms of independent Poisson random variables (r.vs) $Z_j$, $1\leq j\leq n$,  given on some
   probability space $(\Omega,{\mathcal F}, P)$ and such that $\E Z_j=1/j$. Namely,
   \begin{equation}
      \nu(n,r)= \exp\bigg\{\sum_{j^=1}^r\frac{1}{j}\bigg\} P\big(\ell_r(\bar Z)=n\big),
      \label{P-ell}
      \end{equation}
   where $\bar Z:=(Z_1,\dots, Z_n)$. In  two trivial cases, we have
   $ \nu(n,1)=1/n!$ and $\nu(n,n)=1$. It is fairly tedious to extract information from the exact formula if $r$ is large.
    Let us discuss asymptotical behaviour as $n\to\infty$.

      A historical overview may be started from the number of involutions in $\S$.  Namely,  in 1955  L. Moser and M. Wyman \cite{MoWy55} proved that
     \begin{equation}
     n!\nu(n,2)= \frac{1}{\sqrt2} n^{n/2} \exp\Big\{- \frac{n}{2}+n^{1/2}- \frac{1}{4}\Big\}\big(1+o(1)\big).\label{nun2}
     \end{equation}

    H.~Wilf included a detailed proof of (\ref{nun2}) into Chapter 5  of his book \cite{Wilf}. However, Exercise 8 in it
     gives an erroneous expression for $r=3$. It shoud be
         \begin{equation}
     n!\nu(n,3)=\frac{n^{2n/3}}{\sqrt3} \exp\Big\{-\frac{2n}{3} +\frac{1}{2} n^{2/3} +\frac{5}{6} n^{1/3}-\frac{5}{18}\Big\} \big(1+o(1)\big).\label{nun3}
     \end{equation}
    As we have been able to check, the last formula without a detailed proof firstly appears in A.N. Timashov's paper \cite{Timash}. Note that his reference to V.N. Sachkov's paper \cite{Sachkov}, in which formulas  (23) and (24) really concern $n!\nu(n,r)$ for an arbitrary $r$, is misleading. These formulas  have been also presented without a proof; containing  a misprint, they  go in contrast to (\ref{nun3})  and even to (\ref{nun2}). We have  to note that,  a  year later, M. Lugo \cite{Lugo}  also gave (\ref{nun3}) leaving for a reader  other cases of $n!\nu(n,r)$. Recently E. Schmutz kindly gave a reference to the manuscript by T. Amdeberhan and V.H. Moll \cite{AmdMoll} dealing with the same problem. Their Theorem 8.1 also contains errors.  We feel obliged to present a correct formula in Theorem 1.

    Let $\Gamma(z)$ be the Euler gamma-function, where $z\in\C$. Avoiding numerous brackets, instead of $O(\cdot)$, we will use a complex quantity $B$, not the same at different places but always bounded by an absolute constant. Otherwise, stressing dependence on a parameter $v$ in an estimate, we will write $O_v(\cdot)$  with the extra index.

        \begin{theorem}\label{Theorem 1}  If $ 2\leq r\leq \log n$, then
\[
   n!\nu(n,r)=
   \frac{1}{\sqrt r} n^{n(1-1/r)}\exp\bigg\{\sum_{N=0}^r   d_{rN} n^{(r-N)/r}\bigg\} \big(1+Bn^{-1/r}\big).
   \]
 Here $d_{r0}=-1+1/r$,
\[
   d_{r,r}=-\frac{1}{r}\sum_{j=2}^r\frac{1}{j}
   \]
and
\[
             d_{rN}= \frac{\Gamma(N+N/r)}{(r-N)\Gamma(N+1) \Gamma(1+N/r)}
\]
if $1\leq N\leq r-1$.\\
\end{theorem}

Our main results are the next two theorems. We prefer to present them  as asymptotic formulas  for the local probability  $P\big(\ell_r(\bar Z)=n\big)$ having the Cauchy integral representation
          \begin{equation}
   P\big(\ell_r(\bar Z)=n\big)={1\over 2\pi i}\int_{|z|=\alpha}\exp\bigg\{\sum_{j=1}^r{z^j-1\over j}\bigg\}
   {dz\over z^{n+1}},
\label{C-integr}
\end{equation}
where  $\alpha>0$ is to be chosen. In the saddle point method, we  take $\alpha=x:=x(n/r)$, where the function $x(u):=x_r(u)$
is the unique positive solution to the equation
\begin{equation}\label{saddle}
\sum_{j=1}^rx(u)^j=ur, \quad u\geq 1.
\end{equation}
Evidently, the problem concerns asymptotical behaviour of the $n$th power series coefficient of the function
 $\exp\Big\{\sum_{j\leq r} z^j/j\Big\}$, belonging to the so-called  Hayman's class of admissible functions
  (see \cite{Hyman}). B. Harris and  L. Schoenfeld \cite{Har+Schon68} extended Hyman's  methodology in obtaining further asymptotical terms.
  In particular, it yields
\begin{equation}
P\big(\ell_r(\bar Z)=n\big)=\frac{Q(x)}{\sqrt{2\pi \lambda(x)}}\left(1+O_r\Big(\frac{1}{n}\Big)\right)
\label{1Pnr}
\end{equation}
 for arbitrary bounded $r$. Here $x=x(n/r)$ and
 \[
Q(z):=\frac{1}{z^n}\exp\left\{\sum_{j=1}^r\frac{z^j-1}{j}\right\},\qquad  \lambda(z):=\sum_{j=1}^r jz^j.
\]
Actually, we owe to E. Schmutz whose Theorem 1 and the facts presented below it in  \cite{Schmutz88} clarify  the use of the general
and fairly complicated expansion given in \cite{Har+Schon68}. A.N.~Timashov \cite{Timash} mentions a  Sachkov's result from 1986, extending formula
(\ref{1Pnr})  for $r=o(\log n)$. Unfortunately, we failed to find a relevant paper.

    The above mentioned results deal with the case when the ratio $n/r$ is large. In addition, there exists
     a vast literature dealing with the case when $n/r$ is small.
     In fact,  the  problem is related to the  limit distribution of the longest cycle length (say, $L_n(\sigma)$) and other statistics of $\sigma\in\S$.
      So,  V.L.~Goncharov's result  \cite{Gonch44} from 1944 shows that
    \[
           \nu(n, n/u)= \frac{1}{n!}\big|\{\sigma\in \S:\; L_n(\sigma)\leq n/u\}\big|=\rho(u)+o(1)
    \]
    uniformly in  $u\geq1$. Here  $\rho(u)$ is the Dickman function defined as the continuous solution
     to the  difference-differential equation
     \[
                   u\rho'(u)+\rho(u-1)=0
      \]
      with the initial condition $\rho(u)=1$ for $0\leq u\leq 1$.
     Since $\rho(u)\leq \Gamma(u+1)^{-1}$,
      the error in the last estimate can  dominate if $u\to\infty$. Theorem 4.13 in \cite{ABT}, applied for permutations, deals with the  relative error.
       Namely, it shows that
     \[
           \nu(n, r)=\rho(u)\big(1+o(1)\big)
    \]
    if $n/r\to u\in(0,\infty)$. As a byproduct of enumeration of elements in an additive arithmetical semigroup missing large factors,
     the last relation (extended to a larger region for $n/r$) has appeared in the first author's paper \cite{EM-LMJ92}.
     The result is contained in Theorem 3 below. The present paper fills up missing details in its sketchy and indirect proof.

\begin{theorem}\label{thm1} As above, let $x=x(n/r)$.  Then
\[
P\big(\ell_r(\bar Z)=n\big)=\frac{Q(x)}{\sqrt{2\pi \lambda(x)}}\left(1+\frac{Br}{n}\right)
\]
provided that $ 1\leq r\leq c n(\log n)^{-1}(\log\log n)^{-2}$,
where $c=1/(12\pi^2 \re)$ and $n\geq 4$.
\end{theorem}

 Actually, the result holds for all $1\leq r\leq n$. The bound for $r$ is left to show the limitations of the applied approach.
We will prove in Lemma \ref{lema11} of Section 5 that
\begin{eqnarray}
  x&=&
  n^{1/r}-\frac{1}{r}  -\sum_{N=2}^{r}\frac{ \Gamma(N+(N-1)/r)}{(N-1)\Gamma(N+1)\Gamma((N-1)/r)}
   n^{-(N-1)/r} \nonumber\\
   &&\quad    + \frac{1}{r}n^{-1+1/r}+\frac{B}{n},
 \label{xnr-Skleid}
 \end{eqnarray}
 if $1\leq r\leq \log n$. Nevertheless, a direct proof of Theorem 1 by the use of (\ref{xnr-Skleid}) and Theorem 2 would be rather involved. We press more on the Lagrange-B\"{u}rmann inversion formulas, instead. Lemmas \ref{2lema} and \ref{4lema} below provide approximations of $x$ and $\lambda(x)$  for larger $r$. Then it is preferable to apply another technique giving even a sharper remainder term. To catch an idea, let us obtain explicit expressions of the main term of the probability examined in Theorem 2.

  For $u>1$, define $\xi=\xi(u)$ as the nonzero solution to the equation
 \begin{equation}
                 \re^\xi=1+u\xi
\label{alpha}
\end{equation}
and put $\xi(1)=0$. Denote also
\[
              I(s)=\int_0^s \frac{\re^v-1}{v} dv, \quad s\in \C.
\]
Let $\gamma$ denote the Euler--Mascheroni constant.

\s
\begin{corollary} \label{1cor} If $n\geq 4$ and
$
     \sqrt{n\log n}\leq r\leq cn(\log n)^{-1}(\log\log n)^{-2},
$
then
  \begin{eqnarray*}
P\big(\ell_r(\bar Z)=n\big)&=& \frac{1}{\sqrt{2\pi rn}}\exp\bigg\{ I\Big(\frac{n}{r}\Big)-\frac{n}{r}\xi\Big(\frac{n}{r}\Big)\bigg\}\bigg(1+
\frac{Bn\log(n/r)}{r^2}+\frac{B r}{n}\bigg)\\
&=&
\frac{\re^{-\gamma}}{r} \rho\Big(\frac{n}{r}\Big)\bigg(1+\frac{Bn\log(n/r)}{r^2}+\frac{B r}{n}\bigg).
\end{eqnarray*}
\end{corollary}

Consequently, the approximation involving the Dickman function holds in a rather wide region for $n/r$. This makes us even more greedy and  motivates in searching for another approach to refine the corollary.

\begin{theorem}\label{thm2} If  $\sqrt{n\log n}\leq r\leq n$ and $n\geq 1$, then
  \[
P\big(\ell_r(\bar Z)=n\big)=\frac{\re^{-\gamma}}{r}\rho\Big(\frac{n}{r}\Big)\bigg(1+\frac{Bn\log(n/r+1)}{r^2}\bigg).
\]
\end{theorem}

Having in mind an analogy with number theory, one can guess that the last approximation using Dickman's function is hardly further extendable.
The next corollary justifies one expression for all possible $r$.

\begin{corollary}\label{2cor} For $1\leq r\leq n$, we have
\begin{equation}
P\big(\ell_r(\bar Z)=n\big)=\frac{Q(x)}{\sqrt{2\pi \lambda(x)}}\left(1+\frac{Br}{n}\right).\label{fullx}
\end{equation}
\end{corollary}

On the other hand, Theorem \ref{thm2} gives a better remainder term than $ {Br}/{n}$ if $n^{2/3}\log^{1/3} n<r\leq n$.

 The paper is organized as follows. Section 2 collects known and new auxiliary properties of the involved functions and the saddle point approximations.
  Theorem \ref{thm1} and Corollary \ref{1cor} are proved in Section 3. Section 4 is devoted to Theorem \ref{thm2} and Corollary \ref{2cor}. A detailed proof of Theorem 1 is given in the last section.

\section{Auxiliary Lemmas}

Throughout the section, we assume that $r\geq 2$ if it is not indicated otherwise. Let  $\xi(u)$,  $\rho(u)$, and
  $ x(u):=x_r(u)$ be the functions defined above for $u\geq 1$. Recall that they are positive and differentiable if $u>1$.
   We will often  use the abbreviation $f=f(u)$ and $f'=f'(u)$ for the values at the point $u$, where $f(v)$, $v>1$,  is any of the involved functions.

\begin{lemma}\label{1lema} If  $u>1$, then $\log u<\xi:=\xi(u)\leq 2\log u$,
\[
\xi=\log u +\log\log (u+2)+\frac{B\log\log (u+2)}{ \log (u+2)}
\]
and
\begin{equation}
\xi':=\xi'(u)=\frac{1}{u}\,\frac{\xi}{\xi-1+1/u}=\frac{1}{u}\exp\bigg\{\frac{B}{\log (u+1)}\bigg\}.
\label{deriv}
\end{equation}
\end{lemma}

\textit{Proof}.  To establish the effective bounds for all $u>1$, it suffices to employ the strictly increasing function $I'(v)$. Indeed, the lower bound follows from the inequality
 \[
      u=I'(\xi(u))=\int_0^1\re^{t\xi} dt> I'(\log u)=\frac{u-1}{\log u}
 \]
 following from $ u\log u-u+1>0$  if $u>1$. To prove the upper estimate, it suffices to repeat the same argument.

 The asymptotical formulas for $\xi(u)$ and its derivative can be found in \cite{AH+GT} or in the book \cite{GT}.

 The lemma is proved.

\begin{lemma}\label{rho-lema} For $u\geq 1$,
\[
\rho(u)=\sqrt{\frac{\xi'}{2\pi}}\exp\big\{\gamma -u\xi +I(\xi)\big\}\Big(1+\frac{B}{u}\Big).
\]
\end{lemma}

\textit{Proof}. This is Theorem 8 in Section III.5.4 of \cite{GT}. The result has been proved by K. Alladi \cite{Alladi}.

\begin{lemma}\label{rholap} Let
\[
\hat\rho(s):=\int_0^\infty \re^{-sv} \rho(v) dv=\exp\left\{\gamma+I(-s)\right\}, \quad s \in \C,
\]
be the Laplace transform of $\rho(v)$, $s=-\xi(u)+i\tau=:-\xi+i\tau$ and $\tau\in\R$. Then
\begin{equation*}
\hat{\rho}(s) =\begin{cases}B \exp\left\{I(\xi)-\tau^2u/2\pi^2\right\} & {\rm if}\; |\tau|\leq\pi,  \\
B\exp\left\{ I(\xi)-u/(\pi^2+\xi^2) \right\} & {\rm if}\; |\tau|>\pi  \end{cases}
\end{equation*}
and
\begin{equation*}
 \hat{\rho}(s)=\frac{1}{s}\bigg(1+\frac{B(1+\xi u)}{s}\bigg)\quad {\rm if}\; |\tau|>1+u\xi.
\end{equation*}
\end{lemma}

\textit{Proof}. This is Lemma 8.2 in Section III.5.4 of \cite{GT}.

\s

Denote $a\wedge b:=\min\{a, b\}$ and $a\vee b:=\max\{a, b\}$ if $a,b\in \R$. Recall that $x:=x(u)$ is the solution to the saddle point equation and $\lambda(x)=\sum_{j=1}^r jx^j$.

\begin{lemma}\label{2lema} If $u\geq 3$, then
   \begin{equation}
          x=\exp\bigg\{{\log\big( u(r\wedge \log u)\big)\over r}\bigg\}\bigg(1+{B\over r}\bigg).
\label{wedge}
\end{equation}

If $3\leq u\leq \re^r$, then
 \begin{eqnarray}
x&=&\exp\left\{\frac{\log\left(u\log u\right)}{r}\right\}\left(1+\frac{B\log\log u}{r\log u}+\frac{B\log u}{r^2}\right)\nonumber\\
&=&
\exp\Big\{{\xi\over r}\Big\}\left(1+\frac{B\log\log u}{r\log u}+\frac{B\log u}{r^2}\right).
\label{xxi}
\end{eqnarray}

Moreover, for $u>1$,
\begin{equation}
   |\lambda(x)/(r^2u)-1|\leq \log^{-1} u.
\label{lambda}
\end{equation}
\end{lemma}

   \textit{Proof}.  By definition,  $x>1$ and  $u\leq x^r\leq ru$ for $u>1$. The well-known property of geometric
   and arithmetic means
\[
    x^{(r+1)/2}=(x^1 x^2\cdots x^r)^{1/r}\leq {1\over r}\sum_{j=1}^r x^j=u
    \]
yields
\begin{equation}
              u^{1/r}\leq x\leq u^{2/(r+1)}\leq u.
\label{urx}
\end{equation}
We have from the definition that
  \begin{equation}
   x^r=1+ru(1- x^{-1}).
\label{xr}
\end{equation}
Consequently, by (\ref{urx}) and by virtue of $1-\re^{-t}\geq t\re^{-t}$ if $t\geq 0$,
\[
   x^r> ru\big(1-\exp\{-(\log u)/r\}\big)\geq u\log u \exp\{-(\log u)/r\}\geq {\rm e}^{-1} u\log u
\]
 provided that $r\geq \log u$. Similarly,
  \[
   x^r\leq 1+ru(1-\exp\{-2(\log u)/r\})\leq 1+2u\log u.
   \]
   The last two inequalities imply
   \begin{equation}
            r\log x =\log (u\log u)+ B
   \label{rlog}
   \end{equation}
   for $r\geq \log u$.

  If $r\leq  \log u$, we  have
 \[
   x^r> ru\big(1-\exp\{-(\log u)/r\}\big)\geq \big(1-\re^{-1}\big) ru.
\]
and $x^r\leq 1+ru$. Now
\[
            r\log x =\log (ur)+ B.
   \]
      The latter and (\ref{rlog}) lead to relation (\ref{wedge}).

 To sharpen (\ref{wedge}) for $3\leq u\leq \re^r$, we  iterate once more and obtain
\begin{eqnarray*}
  r\log x &=&\log \Big[1+ru\big(1-x^{-1}\big)\Big]\\
   &=&
     \log\bigg[1+ r  u\bigg(1-\exp\bigg\{{-\log( u\log u)\over r}\bigg\}\Big(1+{B\over r}\Big)\bigg)\bigg]\\
     &=&
         \log\Big( u\log( u\log u)+Bu+B(u/r)\log^2 u\Big)\\
     &=&
     \log( u\log u)+{B\log\log u\over \log u}+{B\log u\over r}.
         \end{eqnarray*}
 This is the first relation in (\ref{xxi}).    Comparing the result and Lemma \ref{1lema}, we have the second one.

To prove (\ref{lambda}), we first observe that
\begin{equation}
\lambda(x)=\frac{rx^{r+1}}{x-1}-\frac{x^{r+1}-x}{(x-1)^2}=\frac{rx^{r+1}-ru}{x-1}= r^2u+\frac{r(x-u)}{x-1}.
\label{Lambda}
\end{equation}
Further,
\[
0\leq \frac{1}{ru} \frac{u-x}{x-1}<\frac{1}{r(x-1)}\leq \frac{1}{\log u},
\]
due to  (\ref{urx}) and  $r(x-1)\geq r(\re^{(\log u)/r}-1)\geq \log u$.

    The lemma is proved.

\s

Using properties of differentiable functions, we improve the  remainder term estimates.

\begin{lemma}\label{4lema}
If $1< u\leq \re^r$, $\xi:=\xi(u)$, and $\xi':=\xi'(u)$, then
 \begin{equation}
x=x(u)= \exp\Big\{\frac{\xi}{r}\Big\}+\frac{B\log (u+1)}{r^2}
\label{xxii}
\end{equation}
and
\begin{equation}
\frac{x'}{x}(u)= \frac{\xi'}{r}\Big(1+\frac{B \log (u+1)}{r}\Big).
\label{xxiii}
\end{equation}
\end{lemma}

   \textit{Proof}.  One may skip the  trivial case when  $r$ is bounded.
   From (\ref{saddle}) and (\ref{lambda}),  for the differentiable function $x(v)$, we have
   \begin{equation}
   0<x'(v)=\frac{rx(v)}{\lambda(x(v))}\leq \frac{x(v)}{rv\big(1-\log^{-1} 3\big)}=\frac{Bx(v)}{rv}
   \label{xderiv}
   \end{equation}
if $v\geq 3$. The same holds if $1\leq v\leq 3$. Indeed, in this case it suffices to apply the trivial estimate
$
   \lambda(x(v))\geq  r^2/2\geq r^2 v/6.
$

  As a function of $v$,  $\exp\big\{\xi(v)/r\big\}$ is also strictly increasing; therefore, given any  $u\geq 1$ and the
  value $\xi=\xi(u)$, we can find $w\geq1$ such that
   \[
   x(w)=\exp\big\{\xi/r\big\}.
   \]
Now
\begin{equation}
x-\exp\big\{\xi/r\big\}=x(u)-x(w)= B(u-w) x'(v),
\label{xu=}
\end{equation}
 where $v$ is a point between the $u$ and $w$, irrespective of their relative position on the real line.

  Using  (\ref{xr}) with  $w$ instead of $u$, we have
  \[
    x^r(w)-1=  \re^\xi-1= rw\big(1-x(w)^{-1}\big)= rw\big(1-\re^{-\xi/r}\big).
    \]
By the definition of $\xi$ and Lemma 1, we obtain from the last relation that $u\xi= rw\big(\xi/r+ B(\xi/r)^2\big)$ with $|B|\leq 1/2$. Hence
 \begin{equation}
     |u-w|\leq w\xi/(2r).
 \label{uminv}
 \end{equation}

 If $u\leq 3$ and $r\geq 1$, then $0.09w<w(1-\xi(3))\leq 2u\leq 6$ and $u-w =Br^{-1}$. Therefore,  estimates (\ref{xderiv}) and (\ref{xu=}) imply
 \[
 x-\exp\big\{\xi/r\big\}=B r^{-2},
 \]
 as desired in (\ref{xxii}).

 If $u\geq 3$, then by virtue of $\xi\sim\log u$ as $u\to\infty$ and $\log u\leq r$, we obtain from (\ref{uminv}) that $|u-w|\leq (3/4)w$ if $r$ is
  sufficiently large. Hence $(4/7) u\leq w\leq 4u$ and $(4/7)u\leq v\leq 4 u$. By Lemma \ref{2lema}, this gives
   $x(v)\leq x(4u)=B$. Formula (\ref{xxii}) again follows from (\ref{xderiv}) and (\ref{xu=}).

 To derive approximation (\ref{xxiii}) of the logarithmic derivative, we use similar arguments. First,
  given   $u\geq 3$,  we define  $y>1$ such that
   $x=\re^{\xi(y)/r}$ and claim that
\begin{equation}
     \xi=\xi(y)\big(1+B/r\big).
     \label{y}
     \end{equation}
     Indeed, if also $u\leq \re^r$, then an observation in the proof of Lemma \ref{2lema} gives us $\xi(y)=r\log x\leq \log (ur)
     \leq (6/5) r$ if $r$ is sufficiently large.
By the definitions and inequalities
\[
  0<\frac{t}{1-\re^{-t}}-1=\frac{t-1+\re^{-t}}{1-\re^{-t}}\leq \frac{t^2/2}{t-t^2/2}\leq \frac{3t}{2}
\]
if $0<t\leq 6/5$, we further obtain
     \begin{equation}
  u=\frac{x}{r} \frac{x^r-1}{x-1}=
     \frac{\re^{\xi(y)}-1} {\xi(y)} \frac {\xi(y)/r}{1-\re^{-\xi(y)/r}}=     y\Big(1+\frac{B\xi(y)}{r}\Big)
     \label{u}
     \end{equation}
     with $0<B\leq 3/2$.
      Hence     $ 15/14\leq (5/14)u< y\leq u $
       and also $ \xi'(v)=B/v=B/y$ for all $v\in[y, u]$, by Lemma \ref{1lema}. Inserting this and (\ref{u})
      into   $ \xi-\xi(y)=(u-y)\xi'(v) $ with some $v\in[y, u]$,  we complete the proof of (\ref{y}).

      Let us keep in mind the bound $y\geq 15/14$ and  return to the logarithmic derivative. It follows from (\ref{Lambda}) and (\ref{xr}) that
\[
\frac{x'}{x}\bigg(\frac{x^r}{x^r-1}-\frac{1}{r(x-1)}\bigg)= \frac{1}{ru}.
 \]
Now, the idea is to rewrite  the quantity in large parentheses via $\xi(y)$, then use inequality (\ref{y}) to approximate it by $\xi$ and $\xi'$.

The inequality $0<t^{-1}-(\re^t-1)^{-1}<1$ applied with $t=\xi(y)/r$ gives $(r(x-1))^{-1}=1/\xi(y)+B/r$; therefore,
\begin{equation}
\frac{x'}{x}  \bigg(\frac{1+y\xi(y)-y}{y\xi(y)}+\frac{B}{r}\bigg)= \frac{1}{ru}.
 \label{logder}
 \end{equation}
 Because of (\ref{deriv}), the first ratio inside the parentheses is $1/(y\xi'(y))$ which, by Lemma \ref{1lema},
 satisfies an inequality
 \[
      \frac{1}{y\xi'(y)}\geq \frac{y\log y-y+1}{y\log y}=:q(y)\geq q\Big(\frac{15}{14}\Big)>0.
 \]
Now using (\ref{u}) and (\ref{y}), we obtain
\[
\frac{x'}{x} = \frac{1}{ru} \frac{y\xi(y)}{1+y\xi(y)-y}\Big(1+\frac{B}{r}\Big)= \frac{1}{ru} \frac{u\xi}{1+u\xi-u}\Big(1+\frac{B\log u}{r}\Big)=
 \frac{\xi'}{r}\Big(1+\frac{B\log u}{r}\Big)
 \]
 if $3\leq u\leq \re^r$.

 In the case $1<u\leq 3$, we have from (\ref{xxii})
 \begin{eqnarray*}
 \lambda(x)&=&\sum_{j=1}^r j\bigg(\re^{\xi/r}+\frac{B}{r^2}\bigg)^j=
 r\sum_{j=1}^r \frac{j}{r}\re^{\xi j/r}+Br=r^2\int_0^1t\re^{t\xi} dt+Br\\
 &=&
 \frac{r^2}{\xi}(u\xi+1-u)+Br= \frac{r^2}{\xi'}+B r.
 \end{eqnarray*}
 Hence
 \[
    \frac{x'}{x}=\frac{r}{\lambda(x)}=\frac{\xi'}{r}\Big(1+\frac{B}{r}\Big).
    \]

The lemma is proved.

We will need an estimate of the following function
\[
T(z):=\int_0^{z}\frac{\re^t-1}{t}\left(\frac{t}{r}\frac{\re^{t/r}}{\re^{t/r}-1}-1\right) dt, \quad z\in \C.
\]

  \begin{lemma}\label{6lema} If $z=\eta+i\tau$, $0\leq \eta\leq \pi r$ and $-\pi r\leq \tau\leq \pi r$, then
\[
\Big|T(z)+\frac{z}{2r}\Big|\leq \frac{4\re^\eta}{r}+\frac{\tau^2}{12r^2}.
\]
\end{lemma}

\textit{Proof}. The well known  theory of Bernoulli numbers $\{b_n\}$, $n\geq 0$, gives us the series
\begin{equation}
 b(z):=    \frac{z}{1-\re^{-z}}=\sum_{n=0}^\infty \frac{b_n(-z)^n}{n!}=1+\frac{z}{2}+2\sum_{k=1}^\infty \frac{(-1)^{k+1}\zeta(2k)}{(2\pi)^{2k}} z^{2k}
 \label{Bern}
 \end{equation}
 converging for $z\in\C$. Here $\zeta(2k)=\sum_{m\geq 1}m^{-2k}\leq \zeta(2)=\pi^2/6$. Hence
\begin{eqnarray}
T(z)&=&\frac{1}{2r}\int_0^{z} \left(\re^{t}-1\right) dt+2\sum_{k=1}^\infty \frac{(-1)^{k+1}\zeta(2k)}{(2\pi r)^{2k}}\int_0^{z}(\re^t-1)t^{2k-1}dt\nonumber\\
&=&\frac{1}{2r}(\re^{z}-z-1)+2\sum_{k=1}^\infty \frac{(-1)^{k+1}\zeta(2k)}{(2\pi r)^{2k}}\left( \re^{z}z^{2k-1}-(2k-1)\int_0^{z}\re^t t^{2k-2}dt\right)
\nonumber\\
&&\quad +2\sum_{k=1}^\infty \frac{(-1)^{k}\zeta(2k)z^{2k}}{2k(2\pi r)^{2k}}.
\label{Ts}
\end{eqnarray}

Under assumed conditions, $|z|^2\leq 2\pi^2 r^2$; therefore, summing up the series, we obtain
\begin{eqnarray*}
\Big|T(z)+\frac{z}{2r}\Big|&\leq&
\frac{\re^\eta}{r}+\frac{2\pi^2}{3}\re^\eta\sum_{k=1}^\infty\frac{|z|^{2k-1}}{(2\pi r)^{2k}}+\frac{\pi^2}{6}\sum_{k=1}^\infty\frac{|s|^{2k}}{k(2\pi r)^{2k}}\nonumber\\
&\leq&
\frac{\re^\eta}{r}+\frac{\re^{\eta}(\eta+|\tau|)}{3 r^2}+\frac{\eta^2+\tau^2}{12r^2}
\leq \frac{\re^\eta}{r}\Big(1+\frac{2\pi}{3}+\frac{\pi}{12}\Big)+\frac{\tau^2}{12r^2}\\
&\leq&
 \frac{4\re^\eta}{r}+\frac{\tau^2}{12r^2}.
\end{eqnarray*}

The lemma is proved.

\section{ Proof of Theorem 2}

The non-standard part of our proof concerns the following trigonometric sum
\[
g_r(t, y):=\sum_{j\leq r}\frac{y^j(\re^{itj}-1)}{j},\quad t\in (-\pi, \pi], \; y>1.
\]
 Its behavior  outside a vicinity of the point $t=0$ is rather complicated; therefore, we consider it in a separate lemma.
Denote
\[
\lambda_k:=\sum_{j=1}^rj^{k-1}x^j,\ k\geq 1,
\]
where $x=x(n/r)$. In particular, $\lambda_1=ur$ and  $\lambda_2=\lambda(x)$.

\begin{lemma}
\label{7lema} If $t\in [-\pi,\pi]$ and $y>1$, then
\begin{equation}
\Re g_r(t,y)\leq -\frac{2}{\pi^2}\frac{y^{r+1}}{r(y-1)}\frac{t^2}{(y-1)^2+t^2}+\frac{2y}{r(y-1)}.
\label{grt}
\end{equation}

If  $1/r\leq |t|\leq \pi$, $x=x(u)$, and $u:=n/r\geq 3$, then
\begin{equation}
\Re g_r(t):=\Re g_r(t,x)\leq -\frac{1}{4\pi^2} \frac{u^{1-4/(r+1)}}{\log^2 u}+\frac{2}{r}+\frac{2}{\log u}.
\label{grt1}
\end{equation}
 \end{lemma}

\textit{Proof}. Observe that
\begin{eqnarray}
\Re\sum_{j=1}^r\frac{y^j(e^{itj}-1)}{j} &\leq & \frac{1}{r} \Re \sum_{j=1}^ry^j(e^{itj}-1)\nonumber\\
&=& \frac{y^{r+1}}{r(y-1)}\left(\Re \frac{e^{it(r+1)}(y-1)}{ye^{it}-1}-1\right)+\frac{y}{r(y-1)}\left(1-\Re \frac{e^{it}(y-1)}{ye^{it}-1}\right)\nonumber\\
&\leq &
\frac{y^{r+1}}{r(y-1)}\left(\frac{y-1}{|ye^{it}-1|}-1\right)+\frac{2y}{r(y-1)}.
\label{grt0}
\end{eqnarray}
If $|t|\leq \pi$, then
\[
\frac{|ye^{it}-1|}{y-1}=\left(1+\frac{2y(1-\cos{t})}{(y-1)^2}\right)^{\frac{1}{2}}\geq\frac{((y-1)^2+(4/\pi^2)t^2)^{\frac{1}{2}}}{y-1}
\]
because of
\begin{equation}
2t^2/\pi^2\leq 1-\cos{t}\leq t^2/2.
\label{cos}
\end{equation}
Using also
\[
\frac{\alpha}{\sqrt{\alpha^2+v^2}}-1\leq -\frac{1}{2}\frac{v^2}{\alpha^2+v^2},\quad  \alpha\geq 0,\; v\in\R,
\]
with $\alpha=y-1$ and $v=(2/\pi) t$,
we obtain
\[
\frac{y-1}{|ye^{it}-1|}-1\leq -\frac{2}{\pi^2}\frac{t^2}{(y-1)^2+t^2}.
\]
Inserting this into (\ref{grt0}),  we complete the proof of inequality (\ref{grt}).

If $y=x$, $1/r\leq |t|\leq \pi$ and $u\geq 3$, we  combine (\ref{grt}) with estimate (\ref{urx}). We have
\[
    \frac{x^{r+1}}{x-1}=n+\frac{x}{x-1}\geq ur
    \]
and
\[
  1<\log u\leq r(x-1)\leq r(u^{2/(r+1)}-1)\leq \frac{2r}{r+1} u^{2/(r+1)}\log u.
  \]

   So,  we obtain
\begin{eqnarray*}
\Re g_r(t)&\leq& -\frac{1}{\pi^2} \frac{u}{r^2(x-1)^2}+\frac{2}{r}\Big(1+\frac{1}{x-1}\Big)\\
&\leq&
-\Big(\frac{r+1}{2\pi r}\Big)^2 \frac{u^{1-4/(r+1)}}{\log^2 u}+\frac{2}{r}+\frac{2}{\log u}.
\end{eqnarray*}.

   Lemma \ref{7lema} is proved.

\s

\textit{Proof of Theorem} \ref{thm1}. As it has been mentioned in the Introduction, it suffices  to examine the case when $r\geq 4$ and $n$ is large.
In the  introduced notation, we have $u\geq c^{-1} (\log n)(\log\log n)^2$ and
\begin{eqnarray}
P\big(\ell_r(\bar Z)=n\big)&=& \frac{Q(x)}{2\pi}\left(\int_{|t|\leq t_0}+\int_{t_0<|t|\leq \pi}\right)\exp\left\{g_r(t)\right\}e^{-itn}dt\nonumber\\
&=:&
\frac{Q(x)}{2\pi}\big(K_1(n)+K_2(n)\big)\label{P-int}
\end{eqnarray}
with  $t_0:=r^{-7/12}n^{-5/12}$.

Expanding the integrand in $K_1(n)$, we use relations
${\re}^{it}=1+it-t^2/2-it^3/6+Bt^4$ if  $ t\in\mathbf{R}$ and
$    \re^{w}=1+B|w| \re^{|w|}$ if $ w\in\mathbf{C}$.
 Consequently, checking that $\lambda_4t_0^4\leq (r^3 n)(r^{-7/3}n^{-5/3})= (r/n)^{2/3}\leq 1$ and using the abbreviation $\lambda:=\lambda_2$,  we obtain
\begin{eqnarray*}
\exp\{g_r(t)\}&=&
\exp\big\{i \lambda_1t-(\lambda/2)t^2-i(\lambda_3/6)t^3+B\lambda_4t^4\big\}\\
&=&\exp\big\{it \lambda_1-(\lambda/2)t^2\big\}\big(1-i(\lambda_3/6)t^3 +B\lambda_3^2 t^6\big)  + B\lambda_4 t^4 \exp\big\{-(\lambda/2)t^2\big\}\\
&=&\exp\big\{it \lambda_1-(\lambda/2)t^2\big\}\big(1-i(\lambda_3/6)t^3\big)+B\big(\lambda_4t^4+\lambda_3^2 t^6\big)\exp\big\{-(\lambda/2)t^2\big\}.
\end{eqnarray*}
Recall that $u=n/r$, $\lambda_1=n$, $\lambda_k\leq r^k u$ if $k\geq 1$, and,  by Lemma \ref{2lema}, $\lambda=\lambda(x)\sim nr$  as $n\to\infty$  because of $u\to\infty$. We now see that
\begin{eqnarray*}
K_1(n)&=&\int_{|t|\leq t_0} \re^{-(\lambda/2)t^2} dt +\frac{B}{\sqrt\lambda}\left(\frac{\lambda_4}{\lambda^2}+\frac{\lambda_3^2}{\lambda^3}\right)\\
&=&
\sqrt{\frac{2\pi}{ \lambda}}-\frac{1}{\sqrt{\lambda}}\int_{|v|> t_0\sqrt{\lambda}} \re^{-v^2/2} dv+\frac{B}{u\sqrt{\lambda}}=
\sqrt{\frac{2\pi}{ \lambda}}+\frac{B}{u\sqrt{\lambda}}.
\end{eqnarray*}

Considering  $K_2(n)$, we first observe that, by virtue of (\ref{cos}),
$
\Re g_r(t)\leq -(2/\pi^2) \lambda t^2
$
if $t_0\leq |t|\leq 1/r$. Therefore, the contribution of the integral over this interval to $K_2(n)$ equals  $B/u\sqrt{\lambda}$.

Further, we apply Lemmas \ref{2lema} and \ref{7lema} to get
  \begin{eqnarray*}
  K_2(n)&=&B\max_{1/r\leq |t|\leq \pi} \big|\exp\big\{g_r(t)\big\}\big| +\frac{B}{u\sqrt{\lambda}}\\
  &=& \frac{B}{\sqrt\lambda} \exp\bigg\{-\frac{1}{4\pi^2} \frac{u^{1-4/(r+1)}}{\log^2 u}+\frac{1}{2}\log u+\log r\bigg\}+\frac{B}{u\sqrt{\lambda}}.
  \end{eqnarray*}
It remains to prove that the quantity in the large curly braces does not exceed $-\log u+B$ if the bounds of $r$ are as in Theorem \ref{thm1}. This is trivial, if
$4\log u>r+1\geq 5$.
If $4\log u\leq r+1$ and $n$ is sufficiently large, we have an estimate
\[
   \frac{1}{4\pi^2} \frac{u^{1-4/(r+1)}}{\log^2 u}\geq  \frac{3 c u}{\log^2 u}\geq  \frac{3\log n(\log\log n)^2}
      {(\log\log n+2\log\log\log n +B)^2}\sim 3\log n
\]
which assures the desired  bound $K_2(n)=B/u\sqrt \lambda$.

   Inserting  the estimates of $K_j(n)$, $j=1,2$, into (\ref{P-int}), we finish the proof of Theorem \ref{thm1}.

\s

\textit{Proof of Corollary}  \ref{1cor}. In the above notation, we can rewrite
\begin{eqnarray}
\log Q(x)&=&
-n\log x+\int_1^x\sum_{j=1}^r t^{j-1} dt
=-n\log x+\int_1^x\frac{t^r-1}{t-1} dt\nonumber\\
&=&
-n\log x+\int_0^{r\log x}\frac{\re^v-1}{v} \frac{v}{r} \frac{dv}{1-\re^{-v/r}}\nonumber\\
&=&
-u r\log x +I(r\log x)+T(r\log x).
\label{Tx}
\end{eqnarray}

If  $1< u\leq \sqrt{ n/\log n}$, then  $u\log (u+1)=Br$ and, by Lemma \ref{4lema},
\[
 r\log x= \xi +\frac{B\log(u+1)}{r},\qquad I(r\log x)-I(\xi)=Bu(\log(u+1))/r.
\]
Thus, by  Lemma \ref{6lema}, we obtain
\[
T(r\log x)\leq 4 x^r/r= B\re^{\xi}/r= B(u\log(u+1))/r.
\]
 Inserting these  estimates  into (\ref{Tx}), we deduce
\begin{equation}
Q(x)= \exp\{-u\xi+I(\xi)\}\Big(1+\frac{Bu\log (u+1)}{r}\Big).
\label{Qx}
\end{equation}
Observe also that, by Lemma \ref{4lema}, a relation
\begin{equation}
\frac{1}{\sqrt{\lambda}}=\sqrt{\frac{x'}{rx}}=\frac{\sqrt{\xi'}}{r}\Big(1+\frac{B\log (u+1)}{r}\Big)
\label{lambdax}
\end{equation}
 holds if $1< u\leq \sqrt{ n/\log n}$.

Now, it suffices to apply the last two relations only for
 $ c^{-1} (\log n)(\log\log n)^2\leq u\leq \sqrt {n/\log n}$.
Taking into account Lemma \ref{rho-lema}, we can present the formula in Theorem \ref{thm1}  in two ways:
\begin{eqnarray*}
P\big(\ell_r(\bar Z)=n\big)
&=&
\frac{\sqrt{\xi'}}{r\sqrt{2\pi}}  \exp\{-u\xi+I(\xi)\}\Big(1+\frac{Bu\log u}{r}+\frac{B}{u}\Big)\\
&=&
\frac{\re^{-\gamma}}{r} \rho(u)\Big(1+\frac{Bu\log u}{r}+\frac{B}{u}\Big).
\end{eqnarray*}

  Corollary \ref{1cor} is proved.

\section{ Proof of Theorem 3}\label{s:4}

The idea is to use the Cauchy integral (\ref{C-integr}) with  $\alpha=y:=\re^{\xi/r}$ which is a good approximation of the saddle point.
Here, as above, $\xi=\xi(u)$ is defined by the relation $\re^\xi=1+u\xi$ for $u>1$ and $\xi(1)=0$. Such a choice  relates $Q(z)$ with  the Laplace transform of Dickman's function.
Namely, if $z=\re^{-s/r}$, $s=-\xi+ir t=:-\xi+i\tau $, and $|t|\leq\pi$, then, as in (\ref{Tx}),
\begin{equation}
    Q\big(\re^{-s/r}\big)=
    \exp\big\{us+I(-s)+T(-s)\big\}=
    \hat\rho(s)\exp\big\{-\gamma+us+T(-s)\big\},
    \label{Qz}
\end{equation}
where $T(-s)$ is the function examined in Lemma \ref{6lema}.

Observe that, under the  conditions of Theorem 2, $1\leq u\leq \sqrt{ n/\log n}$, where $n$ may be considered large.
Let us introduce the following vertical line segments in the complex plane:
\[
\Delta_0:=\{s=-\xi+i\tau:\; |\tau|\leq \pi\},\qquad  \Delta_1:=\{s=-\xi+i\tau:\; \pi\leq \tau\leq r\pi\},
\]
\[
 \Delta_2:=\{s=-\xi+i\tau:\; -\pi r\leq \tau\leq -\pi\},\qquad \Delta=\{s=-\xi+i\tau:\; |\tau|\leq r\pi\},
 \]
 and $\Delta_\infty=\{s=-\xi+i\tau:\; |\tau|\geq r\pi\}$. Taking into account (\ref{Qz}), we have from (\ref{C-integr})
 \begin{eqnarray*}
P\big(\ell_r(\bar Z)=n\big)&=&\frac{1}{2\pi i}\int_{|z|=y} \frac{Q(z) dz}{z}\nonumber\\
&=&
\frac{\re^{-\gamma}}{r}\frac{1}{2\pi i}\int_{\Delta}\re^{us}\hat\rho(s) ds+
\frac{\re^{-\gamma}}{2\pi r i}\int_{\Delta}\re^{us}\hat\rho(s)\big(\re^{T(-s)}-1\big) ds\\
&=:&
I+J.
\end{eqnarray*}

Using Lemmas \ref{1lema}, \ref{rho-lema}, and  \ref{rholap} for the case $|\tau|\geq \pi r>1+u\xi$, we obtain
\begin{eqnarray*}
I&=&\frac{\re^{-\gamma}\rho(u)}{r}- \frac{1}{2\pi i r u} \int_{\Delta_\infty}\hat\rho(s) d(\re^{us})\\
&=&
\frac{\re^{-\gamma}\rho(u)}{r}+ \frac{B\re^{-u\xi}}{ur^2} + \frac{1}{2\pi i u r}\int_{\Delta_\infty}\re^{us}\hat\rho(s)\frac{\re^{-s}-1}{s} ds\\
&=&
\frac{\re^{-\gamma}\rho(u)}{r}+ \frac{B\re^{\xi-u\xi}}{ur^2} =\\
&=&
\frac{\re^{-\gamma}\rho(u)}{r}+ \frac{B\rho(u) \re^{\xi-I(\xi)}}{r^2} =\\
&=&
\frac{\re^{-\gamma}\rho(u)}{r}\Big(1+\frac{B}{r}\Big).
\end{eqnarray*}
In the last step, we have used the fact that $I(\xi)\sim \re^\xi/\xi$ as $\xi\to\infty$.

 The next task is to estimate $J$. If $s\in\Delta$ then, by Lemma \ref{6lema}, $T(-s)=B$ and  $\exp\{T(-s)\}=1+B T(-s)$. Let us split $J$
 into the sum of three integrals $J_k$  over the strips $\Delta_k$, where  $k=0,1$ and 2, respectively. If  $s\in \Delta_0$ then  $T(-s)=B(1+u\log u)/r$.
  Therefore, using Lemmas \ref{1lema}, \ref{rho-lema}, and  \ref{rholap}, now for the case $|\tau|\leq \pi$, we derive
\begin{eqnarray*}
   J_0&=&\frac{B(1+u\log u)}{r^2}\int_{\Delta_0} \big|\hat\rho(s)\re^{us}\big||ds|\\
   &=&
   \frac{B(1+u\log u)\rho(u)\sqrt u}{r^2}
   \int_{-\pi}^{\pi}\re^{-\tau^2u/(2\pi^2)} d\tau\\
&=&
   \frac{B(1+u\log u)\rho(u)}{r^2}.
\end{eqnarray*}

Further,
\begin{eqnarray*}
   J_1&=&\frac{1}{2\pi i u r}\int_{\Delta_1} \hat\rho(s)\big(\re^{T(-s)}-1\big) d \re^{us}\\
   &=&
   \frac{B\re^{-u\xi}}{u r}\big|\hat\rho(-\xi+\pi i)T(\xi-\pi i)\big| +\frac{B\re^{-u\xi}}{u r}\big|\hat\rho(-\xi+\pi r i)T(\xi-\pi r i)\big|\\
   &&\quad +
   \frac{B}{u r}\int_{\Delta_1} \re^{us}\Big(\hat\rho(s)'\big(\re^{T(-s)}-1\big)- \hat\rho(s)T'(-s)\re^{T(-s)}\Big)d s\\
   &=:& L_1+L_2+\frac{B}{u r}L_3.
   \end{eqnarray*}

To estimate $L_1$, we combine the first estimate of $\hat\rho(s)$ given in Lemma \ref{rholap} with  Lemmas~\ref{1lema} and \ref{rho-lema}. So we obtain
  \[
      L_1=\frac{B(1+u\log u)}{u r^2}\re^{-u\xi+I(\xi)}=\frac{B\rho(u)(1+u\log u)}{r^2}.
\]
Similarly, the second estimate in Lemma \ref{rholap} leads to
  \[
      L_2=\frac{B\re^{-u\xi}}{u r^2}=\frac{B\rho(u)\re^{-I(\xi)}}{r^2\sqrt u}=\frac{B\rho(u)}{r^2}.
      \]

  Estimation of the integral $L_3$ is more subtle. It uses an estimate
  \[
 1- b(-s/r)-T(-s)= B\left(\frac{\re^\xi}{r}+\left|\frac{s}{r}\right|^2\right)
  \]
  following from  Lemma \ref{6lema} and the asymptotic formula  $ b(v)=1+v/2+Bv^2 $  for  $|v|\leq \pi\sqrt2$.
 We have
  \begin{eqnarray*}
L_3&=&
\int_{\Delta_1}\re^{us} \frac{\re^{-s}-1}{s}\hat{\rho}(s)\left(1+\frac{s}{r(1-\re^{s/r})}\re^{T(-s)}\right) ds\\
&=&
\int_{\Delta_1}\re^{us} \frac{\re^{-s}-1}{s}\hat{\rho}(s)\left(1- b(-s/r)\re^{T(-s)}\right)ds\\
&=&
\int_{\Delta_1}\re^{us} \frac{\re^{-s}-1}{s}\hat{\rho}(s)\bigg(\big(1- b(-s/r)-T(-s)\big)+B\Big(\frac{s T(-s)}{r}+T(-s)^2\Big)\bigg)
ds\\
&=&
B\re^{-u\xi}\int_{\Delta_1} \frac{|\re^{-s}-1|}{|s|}|\hat{\rho}(s)|\Big(\frac{e^\xi}{r}+\frac{|s|^2}{r^2}\Big)
|ds|.
\end{eqnarray*}
Using the  two different estimates of $\hat\rho(s)$ on the line segments $\Delta_{11}:=\{s\in\Delta_1:\; |\Im s|\leq 1+u\xi\}$ and
 $\Delta_{12}:=\Delta_1\setminus \Delta_{11}$ given by Lemma \ref{rholap}, we proceed as follows:
 \begin{eqnarray*}
L_3&=&
B\exp\bigg\{ -u\xi+I(\xi)-\frac{u}{\pi^2+\xi^2} +\xi\bigg\}\int_{\Delta_{11}}\frac{1}{|s|} \bigg(\frac{\re^{\xi}}{r}+ \frac{|s|^2}{r^2}\bigg)|d s|\\
&&\qquad
+B\exp\big\{-u\xi +\xi\big\} \int_{\Delta_{12}}\frac{1}{|s|^2} \bigg(\frac{\re^{\xi}}{r}+ \frac{|s|^2}{r^2}\bigg)|d s|\\
&=&
B\exp\bigg\{ -u\xi+I(\xi)-\frac{u}{\pi^2+\xi^2} +2\xi\bigg\}\frac{1+\xi}{r}+\frac{B\exp\big\{-u\xi +\xi\big\}}{r}\\
&=&
\frac{B\rho(u)\sqrt u\log(u+2)}{r}.
\end{eqnarray*}

Collecting the obtained estimates, we obtain
\[
   J_1= L_1+L_2+\frac{B}{u r}L_3=\frac{B\rho(u)(1+u\log u)}{r^2}.
   \]
   The same holds for integral $J_2$. Consequently,
 \[
 P\big(\ell_r(\bar Z)=n\big)=I+J_0+J_1+J_2=
\frac{\re^{-\gamma}\rho(u)}{r}\Big(1+\frac{B(1+u\log u)}{r}\Big).
\]

Theorem 3 is proved.

\smallskip

{\it Proof of Corollary} \ref{2cor}. Combine Theorems 2 and 3 with relations (\ref{Qx}) and (\ref{lambdax}) valid in the region which is not covered by Theorem 2.

\section{ Proof of Theorem 1}\label{s:5}
Most of the lemmata of this section are well known and could be found in the literature.
Let $\C[[u]]$ be the set of formal power series over the field $\C$ and let $[u^n] g(u)$ denote the $n$th coefficient of $g(u)\in \C[[u]]$ where $n\in\N_0$.

\begin{lemma} \label{lema1}
Let
\[
\Phi(u)=\sum_{N=0}^\infty\Phi_N u^N
\]
be a power series in $\C[[u]]$ with $\Phi_0=1$. Then, the equation $u=z\Phi(u)$ admits a
unique solution
 \[
u=f(z)=\sum_{N=1}^\infty f_Nz^N, \quad   f_N=\frac{1}{N} [u^{N-1}]\Phi(u)^N, \quad N\geq 1.
\]
\end{lemma}

\textit{Proof}. This is Lagrange-B\"{u}rmann Inversion Formula, presented, for instance on page 732 of a fairly concise book \cite{Fl-Sed}.

\smallskip

In this note, superpositions of series involving $f(z)$ are used, therefore we recall more variants of the inversion formula. Let us stress that, by Lemma \ref{lema1}, $f_1=1$; therefore,  $z/f(z)$ and $\log \big(z/f(z)\big)$ have  formal power series expansions.

\begin{lemma} \label{lema2} Let $f(z)$ be as in Lemma $\ref{lema1}$ and  $j\in \N$. Then
\[
   [z^N]\Big(\frac{z}{f(z)}\Big)^j= \frac{j}{j-N}[u^N]\Phi(u)^{N-j}
   \]
 if $N\in \N_0\setminus\{j\}$ and
\[
            [z^j]\Big(\frac{z}{f(z)}\Big)^j=-[u^{j-1}]\Big(\frac{\Phi'}{\Phi}(u)\Big).
                        \]
Moreover,
\[
   [z^N]\log \frac{z}{f(z)}= -\frac{1}{N}[u^N]\Phi(u)^{N}
   \]
   if $N\geq1$.
\end{lemma}

\textit{Proof}. Without a proof the first part of Lemma \ref{lema2}  is exposed as \textbf{A.11} on pages 732-733 of  \cite{Fl-Sed}; an inaccuracy is left in the case $N=j$, however.  For readers convenience,   we provide a sketch of a proof.

    Let $N\in\N_0$ be fixed. The coefficients under consideration have expressions in terms of $\Phi_k$ with $0\leq k\leq N$ only; therefore, we may assume that $\Phi(u)$ is a polynomial of degree $N$. Then $f(z)$ is well defined as an analytic function in a vicinity of the zero point. Thus, we may  apply Cauchy's formula. Afterwards let  $\delta$ and $\delta_1$ be sufficiently small positive constants. Using a substitution $z=u/\Phi(u)$ and properties of the one-to-one conformal mapping of the vicinities of the zero points in the $z$- and $u$-complex planes,  we obtain
\begin{eqnarray*}
[z^N]\Big(\frac{z}{f(z)}\Big)^j&=&
\frac{1}{2\pi i}\int_{|z|=\delta}\frac{dz}{f(z)^j z^{N+1-j}}\\
&=&
\frac{1}{2\pi i}\int_{|u|=\delta_1}\frac{d\big(u/\Phi(u)\big)}{u^j\big(u/\Phi(u)\big)^{N+1-j}}\\
&=&
\frac{1}{2\pi i}\int_{|u|=\delta_1}\frac{\Phi(u)^{N-j} du}{u^{N+1}}-\frac{1}{2\pi i}\int_{|u|=\delta_1}\frac{\Phi(u)^{N-j-1} d\Phi(u)}{u^N}\\
&=&
[u^N]\Phi(u)^{N-j}-\frac{1}{2\pi (N-j)i}\int_{|u|=\delta_1}\frac{d\Phi(u)^{N-j}}{u^{N}}\\
&=&
[u^N]\Phi(u)^{N-j}-\frac{N}{2\pi (N-j)i}\int_{|u|=\delta_1}\frac{\Phi(u)^{N-j} du}{u^{N+1}}\\
&=&
\frac{j}{j-N}[u^N]\Phi(u)^{N-j}
\end{eqnarray*}
provided that $N\not=j$.

The same argument gives
\begin{eqnarray*}
[z^j]\Big(\frac{z}{f(z)}\Big)^j&=&
\frac{1}{2\pi i}\int_{|u|=\delta_1}\frac{d u}{u^{j+1}}-\frac{1}{2\pi i}\int_{|u|=\delta_1}\frac{\Phi'}{\Phi}(u) \frac{du}{u^j}\\
&=&
-[u^{j-1}]\Big(\frac{\Phi'}{\Phi}(u)\Big).
\end{eqnarray*}

Finally, applying the previous substitution, we derive
\begin{eqnarray*}
[z^N]\log \frac{z}{f(z)}&=&
\frac{1}{2\pi N i }\int_{|z|=\delta}\frac{1}{z^{N}}d\log\frac{z}{f(z)} \\
&=&
-\frac{1}{2\pi N^2i}\int_{|u|=\delta_1}\frac{ d\Phi(u)^N}{u^N}\\
&=&
-\frac{1}{N}[u^N] \Phi(u)^N.
\end{eqnarray*}

The lemma is proved.

We will apply the lemmas in a very particular case. Then the first power series coefficients of implicitly defined functions attain a simple form.
 Let ${\mathbf 1}\{\cdot\}$ stand for the indicator function.

\begin{lemma}\label{lema3}    Let $k,r, j\in\N$, $y=y(z)$ satisfy an equation
 \[
 y=z\bigg(\frac{1-y^r}{1-y}\bigg)^{1/r},
 \]
 and let $g(z):=z/y(z)$, then  the following assertions hold.

   $(I)$ \quad If $g(z)^j=:\sum_{N=0}^\infty g_N^{(j)} z^N$, then
      \[
  g_N^{(j)}= \frac{j}{j-N}\sum_{rl+m=N\atop l,m\in\N_0} {(N-j)/r\choose l}(-1)^l {m-1+(N-j)/r\choose m}
   \]
   for $N\in\N_0\setminus\{j\}$ and
  \begin{equation}
  g_j^{(j)}={\mathbf 1}\{j\equiv 0({\rm mod}\, r)\}-\frac{1}{r}.
   \label{gjj}
\end{equation}

   $(II)$ \quad If $\log g(z)=:\sum_{N=1}^\infty b_N z^N$, then
  \[
  b_N=-\frac{1}{N}\sum_{rl+m=N\atop l,m\in\N_0} {N/r\choose l}(-1)^l {m-1+N/r\choose m}, \quad N\geq1.
    \]

 $(III)$ \quad If
 \[
      h(z):=\sum_{j=1}^r\frac{1}{j y(z)^j}=\sum_{N=-r}^\infty h_N z^N,
      \]
      then
       $h_{-r}=1/r$,
\[
  h_0=- \frac{1}{r}\sum_{j=2}^r\frac{1}{j}
\]
and
\[
h_N= \frac{N+r}{N} b_{N+r}
\]
for  $N=-r+1,-r+2,\dots$ and $N\not=0$.

 $(IV)$ \quad If
 \[
      \Lambda(z):=\bigg(z^r\sum_{j=1}^r\frac{j}{y(z)^j}\bigg)^{-1}=\sum_{N=0}^\infty \Lambda_N z^N
      \]
then $\Lambda_0=1/r$ and $\Lambda_N=-Nb_N/r$ for $N=1, 2,\dots$.
    \end{lemma}

   \textit{Proof}. To prove $(I)$, combine Lemmas \ref{lema1} and  \ref{lema2} with  an equality
     \[
    [y^{N}]\bigg(\frac{1-y^r}{1-y}\bigg)^{\alpha}=\sum_{rl+m=N\atop l,m\in\N_0} {\alpha\choose l}(-1)^l {m-1+\alpha \choose m}, \quad N\in\N_0,\, \alpha\in\R.
   \]

For $(\ref{gjj})$, apply the second part of Lemma \ref{lema2} to obtain
\[
   g_j^{(j)}=[y^{j-1}]\bigg(\frac{y^{r-1}}{1-y^r}-\frac{1}{r(1-y)}\bigg)={\mathbf 1}\{j\equiv 0({\rm mod}\, r)\}-\frac{1}{r}.
\]

   Similarly, $(II$) follows from the last formula in Lemma \ref{lema2}.

Having in mind that $z^rh(z)$ has a power series expansion in $\C[[z]]$, we may apply the same principles. Using $(\ref{gjj})$, it is easy to check that
\begin{eqnarray*}
h_0&=&\frac{1}{2\pi  i }\int_{|z|=\delta}\frac{ h(z) dz}{z} =
\sum_{j=1}^r\frac{1}{j} \frac{1}{2\pi  i }\int_{|z|=\delta}\frac{ g(z)^j d z}{z^{j+1}}\\
&=&
\sum_{j=1}^r\frac{1}{j} g_j^{(j)}=- \frac{1}{r}\sum_{j=2}^r\frac{1}{j}.
\end{eqnarray*}
Further, we  observe that
\begin{equation}
   \sum_{j=1}^r \frac{1}{y^j}=\frac{1-y^r}{y^r(1-y)}=\frac{1}{z^{r}}.
\label{frac}
\end{equation}
Hence
\[
      h'(z)=-\frac{y'}{y}(z) \sum_{j=1}^r \frac{1}{y^j}=-\frac{y'}{y}(z) \frac{1}{z^{r}}= \frac{g'}{g}(z) \frac{1}{z^{r}}-\frac{1}{z^{r+1}}, \quad z\not=0.
\]
This implies that
\begin{eqnarray*}
h_N&=&\frac{1}{2\pi N i }\int_{|z|=\delta}\frac{ d h(z)}{z^{N}} \\
&=&
\frac{1}{2\pi N i }\int_{|z|=\delta}\frac{ d (\log g(z))}{z^{N+r}}-\frac{1}{2\pi N i }\int_{|z|=\delta}\frac{ d z}{z^{N+r+1}} \\
 &=& \frac{N+r}{N} b_{N+r}
\end{eqnarray*}
if $N\geq -r+ 1$ and $h_{-r}=1/r$.

To prove $(IV)$, we use relation (\ref{frac}) again. Differentiating it, we arrive at
\begin{eqnarray*}
          \Lambda(z)&=&\frac{z}{r} \frac{y'}{y}(z)=\frac{1}{r}-\frac{z\big(\log g(z)\big)'}{r}\\
          &=&\frac{1}{r}\bigg(1-\sum_{N=1}^\infty Nb_N z^N\bigg).
\end{eqnarray*}
The assertion $(IV)$ now is evident.

   The lemma is proved.

\begin{corollary}\label{cor1} As above, let $g(z)=z/y(z)$. Then $g_0=1$, $g_1=-1/r$,
    \begin{equation*}
   g_N=  \frac{ \Gamma(N+(N-1)/r)}{(1-N)\Gamma(N+1)\Gamma((N-1)/r)}
   \end{equation*}
   if $2\leq N\leq r-1$, and
   \[
   g_r=  \frac{ \Gamma(r+1-1/r)}{(1-r)\Gamma(r+1)\Gamma(1-1/r)}+\frac{1}{r}.
   \]
   Moreover,
   $
    |g_N|\leq \frac{1}{N-1} r^{(N-1)/r}
$
if $N\geq 2$.
     \end{corollary}

  \textit{Proof}. Apply $(I)$ of Lemma \ref{lema3} for $j=1$. If $2\leq N\leq r-1$, the relevant sum has the only nonzero summand corresponding to the pair $(l,m)=(0,N)$. A formula for $g_r$ has two summands giving the expression. If $N\geq 2$, then by Lemma \ref{lema3} and Cauchy's inequality,
  \[
    |g_N|=  \frac{1}{N-1}\Big|[y^N](1+y+\cdots+y^{r-1})^{(N-1)/r}\Big|\leq \frac{1}{N-1} r^{(N-1)/r}.
\]

The corollary is proved.

\begin{corollary} \label{cor2} We have
      \begin{equation}
  b_N=- \frac{\Gamma(N+N/r)}{N\Gamma(N+1)\Gamma(N/r)}
  \label{bN}
  \end{equation}
  if $1\leq N\leq r-1$ and $b_r=0$.

  Moreover,
  \[
  N|b_N|\leq 1 \quad  if\:  N\leq r-1,
  \]
  \[
  b_N =BN/r \quad  if\:  r<N\leq 2 r-1,
  \]
and
   \begin{equation*}
          N|b_N|\leq  r^{N/r}\quad    if\; N\geq 1.
       \end{equation*}
     \end{corollary}

  \textit{Proof}. Again, if $1\leq N\leq r-1$,  it suffices to observe that the relevant sum $(II)$ of  Lemma \ref{lema3} has the only nonzero summand corresponding to the pair $(l,m)=(0,N)$. A formula for $b_r$ has two subtracting summands.

If $N\leq r-1$, the given estimate follows from (\ref{bN}). If $r<N\leq 2r-1$, assertion $(II)$ in Lemma \ref{lema3} gives
     \begin{eqnarray*}
   b_N&=&
   -\frac{1}{N}{N-1+N/r\choose N}+\frac{1}{r}{N-r-1+N/r\choose N-r}\\
   &=&
   -\frac{1}{r} \prod_{k=2}^N\Big(1+\frac{N/r-1}{k}\Big)+\frac{N}{r^2} \prod_{k=2}^{N-r}\Big(1+\frac{N/r-1}{k}\Big)\\
   &=&
       \frac{B}{r}\exp\bigg\{\Big(\frac{N}{r}-1\Big)\sum_{k=2}^N \frac{1}{k}\bigg\}=\\
       &=&
    \frac{B}{r}\exp\bigg\{\Big(\frac{N}{r}-1\Big)\log N\bigg\}=  \frac{ BN}{r}.
   \end{eqnarray*}
   We have applied an inequality $\log(1+x)\leq x$ if $x>0$.

   Finally, by Cauchy's inequality,
   \begin{equation}
          N|b_N|=
     \bigg| [y^{N}]\bigg(\frac{1-y^r}{1-y}\bigg)^{N/r}\bigg|=
      \big| [y^{N}](1+y+\cdots+ y^{r-1})^{N/r}\big|\leq
       r^{N/r}
       \label{NbN}
       \end{equation}
       if $N\geq 1$.

   The corollary is proved.

We now prove the promised expansion (\ref{xnr-Skleid}).

\begin{lemma} \label{lema11}
If $2\leq r\leq \log n$, then
 \begin{eqnarray*}
  x&=&
  n^{1/r}-\frac{1}{r}  -\sum_{N=2}^{r}\frac{ \Gamma(N+(N-1)/r)}{(N-1)\Gamma(N+1)\Gamma((N-1)/r)}
   n^{-(N-1)/r}\\
   &&\quad    + \frac{1}{r}n^{-1+1/r}+\frac{B}{n}.
 \end{eqnarray*}
\end{lemma}

\begin{proof}
The equation  defining  $x$ can be rewritten as
\[
    x^{-1}=\bigg(\frac{1-x^{-r}}{1-x^{-1}}\bigg)^{1/r}n^{-1/r}.
\]
This gives the relation $y(n^{-1/r})=x^{-1}$, where $y=y(z)$ has been explored in Lemma \ref{lema3}. Consequently, we may apply the expansions of $g(z)$ given in $(I)$  with respect to powers of  $z=n^{-1/r}$.
The first coefficients have been calculated in Corollary \ref{cor1}. It remains to estimate the remainder. Using also the obtained estimates, we have
\[
   \sum_{N=r+1}^\infty |g_N| |z|^N\leq r^{-1-1/r}\sum_{N=r+1}^\infty  |r^{1/r} z|^N\leq \frac{|z|^{r+1}}{1-\sqrt[3]{3} {\rm e}^{-1}}
\]
if $|z|\leq {\rm e}^{-1}$.  Consequently,  we obtain
  \begin{eqnarray*}
  x&=&n^{1/r} \sum_{N=0}^r g_N n^{-N/r}+ \frac{B}{n}\\
  &=&
  n^{1/r}-\frac{1}{r}  -\sum_{N=2}^{r}\frac{ \Gamma(N+(N-1)/r)}{(N-1)\Gamma(N+1)\Gamma((N-1)/r)}
   n^{-(N-1)/r} \\
   &&\quad
    + \frac{1}{r}n^{-1+1/r}+\frac{B}{n}
  \end{eqnarray*}
  as desired.
\end{proof}

\textit{Proof of Theorem 1.} Let us preserve the  notation  introduced in Lemma \ref{lema3}.
First of all we seek  a simple expression containing the first terms in an expansion of
\[
  K(z):=   \sum_{j=1}^r \frac{1}{j y(z)^j}-n\log \frac{z}{y(z)}= h(z)-n\log g(z).
\]
Let $D(x):=\exp\left\{\sum_{j=1}^r\frac{x^j}{j}\right\}$, we have
\begin{equation}
\log D(x)-n\log x= K(n^{-1/r})-\frac{n\log n}{r}.
\label{Dx}
\end{equation}
Define the  functions $R(z)$ and $K_r(z)$ by
\[
K(z)=\sum_{N=-r+1}^0 h_N z^N-n\sum_{N=1}^{r-1} b_N z^N +R(z)=K_r(z)+R(z)
\]
We claim  that $R(z)=B(|z|+n|z|^{r+1})$ if $|z|\leq \rm e^{-1}$ implying
\begin{equation}
  R(n^{-1/r})=Bn^{-1/r}.
  \label{Rz}
  \end{equation}
 for $r\leq \log n$. Indeed, by $(III)$ of Lemma \ref{lema3} and the estimates in Corollary \ref{cor2}, we have
       \begin{eqnarray*}
          \sum_{N=1}^\infty |h_N||z|^N&=&
           \bigg(\sum_{N=1}^{r-1}+\sum_{N=r}^\infty\bigg)\frac{N+r}{N}|b_{N+r}||z|^N\\
          &=&
          B\sum_{N=1}^{r-1}\frac{N+r}{r}|z|^N +B\sum_{N=r}^\infty\frac{r}{N}\big(r^{1/r}|z|\big)^N
          =B|z|
          \end{eqnarray*}
       if $|z|\leq \rm e^{-1}$. Similarly,
       \[
          \sum_{N=r+1}^\infty |b_N||z|^N=B |z|^{r+1}
          \]
       if $|z|\leq \rm e^{-1}$. The last two estimates yield our claim and (\ref{Rz}).

     For the main term, we obtain from  Lemma \ref{lema3} that
     \begin{eqnarray*}
     K_r(n^{-1/r})&=& h_0+\sum_{N=-r+1}^{-1} h_N n^{-N/r}- \sum_{N=1}^{r-1} b_N n^{(r-N)/r}+h_{-r} n\\
     &=&
     h_0 -\sum_{N=1}^{r-1} \frac{r-N}{N} b_{r-N} n^{N/r}- \sum_{N=1}^{r-1} b_N n^{(r-N)/r}+h_{-r} n\\
     &=&
     h_0 -r\sum_{N=1}^{r-1} \frac{1}{N} b_{r-N} n^{N/r}+h_{-r} n\\
     &=&
- \frac{1}{r}\sum_{j=2}^r\frac{1}{j}+ r \sum_{N=1}^{r-1}\frac{1}{N(r-N)}\, \frac{\Gamma(N+N/r)}{\Gamma(N+1)\Gamma(N/r)}n^{(r-N)/r}+\frac{n}{r}.
     \end{eqnarray*}

     It remains to approximate
     \[
     \Big(\frac{1}{\lambda(x)}\Big)^{1/2}= \frac{1}{\sqrt n} \Lambda(n^{-1/r})^{1/2}=\frac{1}{\sqrt {n r}}\bigg(1-\sum_{N=1}^\infty Nb_N n^{-N/r}\bigg)^{1/2}.
     \]

    By virtue of Corollary \ref{cor2}, $N|b_N|\leq 1$ if $N\leq r$ and  $N|b_N|\leq r^{N/r}$ if $N\geq 1$.
       Thus, if $2\leq r\leq \log n$, then
     \[
        \sum_{N=1}^\infty N|b_N|n^{-N/r}\leq (5/2) n^{-1/r}\leq (5/2){\rm e}^{-1}<1.
        \]
 Consequently,
\[
     \Big(\frac{1}{\lambda(x)}\Big)^{1/2}= \frac{1}{\sqrt {n r}}\big(1+Bn^{-1/r}\big).
\]

     We  now  return to probabilities.
     Applying \eqref{P-ell}, \eqref{fullx},   (\ref{Dx}),   (\ref{Rz}), Stirling's formula and the last estimate,  we obtain
\begin{eqnarray*}
   n!\nu(n,r)&=&
   \frac{n!}{{\sqrt{ 2\pi \lambda(x)}}} n^{-n/r}\exp\big\{K_r(n^{-1/r})\big\}\big(1+Bn^{-1/r}\big)\\
   &=&
   \frac{n^{n(1-1/r)}}{\sqrt r} \exp\big\{-n + K_r(n^{-1/r})\big\}\big(1+Bn^{-1/r}\big)
   \end{eqnarray*}
for  all $2\leq r\leq \log n$.

Theorem 1 is proved.

\medskip

  {\bf Concluding Remark.} The approach can be adopted for more general decomposable structures, in particular, for the so-called logarithmic classes of set constructions (see \cite{ABT}). Using a different method, X. Gourdon \cite{Gourdon} has established some results related to our Theorem 2. They concern the asymptotic distribution of size of the largest component.

\bigskip
 Vilnius University,
  Institute and Faculty of Mathematics and Informatics,
Akademijos str. 4, LT-08663 Vilnius, Lithuania;

\noindent e-mail address:\textit{eugenijus.manstavicius}.at.\textit{mif.vu.lt}
\smallskip

 Vilnius University,
  Faculty of Mathematics and Informatics,
  Naugarduko str. 24, LT-03225 Vilnius, Lithuania;

  \noindent e-mail address:\textit{robertas.petuchovas}.at.\textit{mif.vu.lt}
\medskip

\end{document}